\documentclass[a4paper,10pt, reqno]{amsart}
\usepackage{latexsym,amsfonts,amsmath, amsthm, amssymb, tikz}
\usetikzlibrary{shapes}
\usetikzlibrary{intersections}

\usepackage{graphics,graphicx}
\usepackage{a4wide}

\newtheorem{thm}{Theorem}
\newtheorem*{lemma}{Lemma}
\newtheorem*{corollary}{Corollary}

\title[A geometric uncertainty principle]{A geometric uncertainty principle with\\ an application to Pleijel's estimate}
\author{Stefan Steinerberger}\thanks{Mathematisches Institut, Universit\"at Bonn, Endenicher Allee 60, 53115 Bonn, Germany, e-mail: \textsc{steinerb@math.uni-bonn.de}}

\begin{document}
\begin{abstract} One cannot decompose a domain into disks of equal radius:
let $\Omega \subset \mathbb{R}^2$ be an open, bounded domain and $\Omega = \bigcup_{i=1}^N \Omega_i$
be a partition. Denote the Fraenkel asymmetry by $0 \leq \mathcal{A}(\Omega_i) \leq 2$ and write
$$ D(\Omega_i) := \frac{|\Omega_i| - \min_{1 \leq j \leq N}{|\Omega_j|}}{|\Omega_i|}$$
with $0 \leq D(\Omega_i) \leq 1$. For $N$ sufficiently large depending only on $\Omega$, there is an uncertainty principle
$$ \left(\sum_{i=1}^{N}{\frac{|\Omega_i|}{|\Omega|}\mathcal{A}(\Omega_i)}\right)+\left(\sum_{i=1}^{N}{\frac{|\Omega_i|}{|\Omega|}D(\Omega_i)}\right) \geq \frac{1}{60000}.$$
The statement remains true in dimensions $n \geq 3$ for some constant $c_n > 0$. As an application, we give an (unspecified) improvement 
of Pleijel's estimate on the number of nodal domains of a Laplacian eigenfunction similar to recent work of Bourgain and improve another
inequality in the field of spectral partition problems.
\end{abstract}

\maketitle

\section{Introduction}
\subsection{Motivation.}  It is easy to partition $\mathbb{R}^2$ into sets of equal measure that are 'almost' disks 
(the hexagonal packing, for example) and it is also possible to decompose $\mathbb{R}^2$ into disks of different size
(Apollonian packings) -- but obviously not both at the same time. We are interested in a quantitative descriptions of 
this phenomenon.
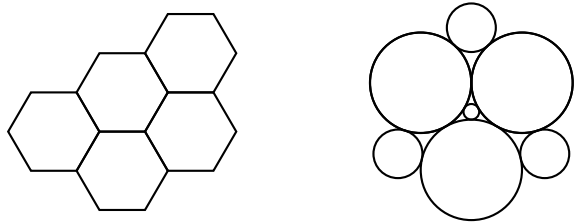
\begin{figure}[h!]
\begin{minipage}[l]{.3\textwidth}
\begin{center}
\begin{tikzpicture}
  \begin{scope}[%
every node/.style={anchor=west,
regular polygon, 
regular polygon sides=6, thick,
draw,
minimum width=1.2cm,
outer sep=0,
},
      transform shape]
    \node (A) {};
    \node (B) at (A.corner 1) {};
    \node (C) at (B.corner 5) {};
    \node (D) at (A.corner 5) {};
    \node (E) at (B.corner 1) {};
   % \node (F) at (A.corner 4) {};
    %\foreach \hex in {A,...,E}
    %{
    %  \foreach \corn in {1,...,6}
    %    \draw[fill=white] (\hex.corner \corn) circle (2pt); 
    %}
  \end{scope}
\end{tikzpicture}
\end{center}
\end{minipage} 
\begin{minipage}[r]{.3\textwidth}
\begin{center}
\begin{tikzpicture}
 \draw[thick] (xyz polar cs:angle=30,radius=0.77) circle (0.666cm);
\draw[thick] (xyz polar cs:angle=150,radius=0.77) circle (0.666cm);
\draw[thick] (xyz polar cs:angle=270,radius=0.77) circle (0.666cm);
\draw[thick] (0,0) circle (0.10333cm);
\draw[thick] (xyz polar cs:angle=30,radius=0.77) circle (0.666cm);
\draw[thick] (xyz polar cs:angle=150,radius=0.77) circle (0.666cm);
\draw[thick] (xyz polar cs:angle=90,radius=1.11467) circle (0.322cm);
\draw[thick] (xyz polar cs:angle=210,radius=1.11467) circle (0.322cm);
\draw[thick] (xyz polar cs:angle=330,radius=1.11467) circle (0.322cm);
\end{tikzpicture}
\end{center}
\end{minipage} 
\caption{A partition into sets of equal measure and 
a partition into disks (only large disks visible).}
\end{figure}

This question turns out to have some relevance in the calculus of variations, in particular in the study of 
vibrations of a membrane $\Omega \subset \mathbb{R}^2$ as well as in spectral partition problems: given an
eigenfunction $\phi$ of the Laplacian $-\Delta$ with Dirichlet boundary conditions on $\Omega$, what is the maximal number of
connected components of $\Omega \setminus \left\{x \in \Omega: \phi(x) = 0\right\}$? Our quantitative study of this simple 
geometric principle in terms of Fraenkel asymmetry and size is very much motivated by the applicability to nodal domain 
estimates -- it could be of interest to capture the same phenomenon in other geometrically natural quantities.

\subsection{Geometric notions.} Let $n \geq 2$. Consider an open, bounded domain $\Omega \subset \mathbb{R}^n$ with a given decomposition
$$ \Omega = \bigcup_{i=1}^{N}{\Omega_i}.$$
We require two quantities to measure
\begin{enumerate}
 \item the deviation of $\Omega_i$ from a ball
 \item the deviation of $|\Omega_i|$ from
$$ \min_{1 \leq j \leq N}{|\Omega_j|}.$$
\end{enumerate}
In measuring how much a set deviates from a ball, Fraenkel asymmetry has recently become an increasingly central 
notion (i.e. \cite{fus2}): given a domain $\Omega \subset \mathbb{R}^n$, its Fraenkel asymmetry is defined via
$$ \mathcal{A}(\Omega) := \inf_{B}{\frac{| \Omega \triangle B|}{|\Omega|}},$$
where the infimum ranges over all disks $B \subset \mathbb{R}^n$ with $|B| = |\Omega|$ and $\triangle$ is the symmetric difference
$$\Omega \triangle B = (\Omega \setminus B) \cup (B \setminus \Omega).$$
Fraenkel asymmetry is scale-invariant 
$$ 0 \leq \mathcal{A}(\Omega) \leq 2.$$
As for deviation in size, we define the deviation from the smallest element in the partition via
$$ D(\Omega_i) := \frac{|\Omega_i| - \min_{1 \leq j \leq N}{|\Omega_j|}}{|\Omega_i|},$$
which is scale invariant as well and satisfies
$$ 0 \leq D(\Omega_i) \leq 1.$$

\subsection{Main result.} Our main result states that for partitions of $\Omega$ into a large number of sets, an \textit{average} element of
the partition needs to have either its Fraenkel asymmetry $\mathcal{A}(\Omega_i)$ or its deviation from the smallest element $D(\Omega_i)$ bounded
away from 0 by a \textit{universal} constant. This statement obviously fails if we only pick one of the two terms: any set can be decomposed
into $N$ sets of measure $|\Omega|/N$ each or each set can be decomposed into disks of different radii with an arbitrarily small measure of different shape
(packings of Apollonian type).

\begin{thm} Suppose $\Omega \subset \mathbb{R}^n$ is an open and bounded domain and 
$$ \Omega = \bigcup_{i=1}^{N}{\Omega_i}$$
with measurable sets $\Omega_i$ satisfying
$$ \Omega_i \cap \Omega_j = \emptyset \qquad \mbox{for} \quad i \neq j.$$
There exists a universal constant $c_n > 0$ depending only on the dimension and a constant $N_0 \in \mathbb{N}$ depending only on $\Omega$ such that for $N \geq N_0$
$$ \left(\sum_{i=1}^{N}{\frac{|\Omega_i|}{|\Omega|}\mathcal{A}(\Omega_i)}\right)+\left(\sum_{i=1}^{N}{\frac{|\Omega_i|}{|\Omega|}D(\Omega_i)}\right) \geq c_n.$$
In particular,
$$c_2 \geq \frac{1}{60000}.$$
\end{thm}

\textbf{Remarks.}
\begin{itemize}
 \item  Taking $\Omega$ to be the union of a finite number of disjoint balls of equal radius shows that such a statement can only hold for $N$ sufficiently large depending on $\Omega$.
 \item There are no assumptions whatsoever on the shape of $\Omega_j$ -- they need not be connected.
\item Fraenkel asymmetry turns the problem into a non-local one as the 'missing' measure $\Omega \triangle B$ can be arbitrarily spread over the plane: this
is why we believe that any argument yielding a substantially improved constant will need to be based on significantly new ideas. Indeed, our proof will essentially
only be a 'non-local perturbation' of a local argument but not truly non-local itself (hence the small constant).
\item What can be said about the optimal constant $c_n$? A natural candidate for an extremizer in $\mathbb{R}^2$ is the hexagonal tiling, 
which suggests that maybe
$$ c_2 \sim 0.074465754\dots$$
As packing density of spheres decreases in higher dimensions, we consider it extremely natural to conjecture that $$c_2 \leq c_3 \leq \dots$$ 
\item The following interesting question is due to Almut Burchard: suppose the hexagonal packing was indeed a minimizer; we can introduce a parameter
$\alpha > 0$ and look for minimizers of
$$ \alpha\left(\sum_{i=1}^{N}{\frac{|\Omega_i|}{|\Omega|}\mathcal{A}(\Omega_i)}\right)+\left(\sum_{i=1}^{N}{\frac{|\Omega_i|}{|\Omega|}D(\Omega_i)}\right).$$
It seems reasonable to conjecture that the hexagonal packing will then be a minimizer for every $0 < \alpha \leq 1$. However, it is easy to see that
there will be some $\alpha_0 \geq 1$ such that the hexagonal packing is no longer minimizing for any $\alpha > \alpha_0$. What happens at the transition?
Which configurations minimize the expression then?
\end{itemize}

\subsection{Variants and extensions.} There are many possible variations and extensions. We can write Fraenkel asymmetry as
$$ \mathcal{A}(\Omega) = \inf_{x \in \mathbb{R}^n}{\frac{| \Omega \triangle (B+x)|}{|\Omega|}},$$ 
where $B$ is the ball centered at the origin scaled in such a way that $|B| = |\Omega|$. However, this definition can be easily generalized by
considering other sets $K$ instead of the ball if one corrects for the arising lack of rotational symmetry, i.e.
$$ \mathcal{A}_{K}(\Omega) := \inf_{x \in \mathbb{R}^n}{\inf_{R \in \mathcal{R}}{\frac{| \Omega \triangle (RK+x)|}{|\Omega|}}},$$ 
where $K$ is scaled in such a way that $|K| = |\Omega|$ and $\mathcal{R}$ is the set of all rotations. The proof of our main statement is quite robust: it immediately
allows to prove the following variant.

\begin{thm} Let $K \subset \mathbb{R}^n$ be a bounded, convex set with a smooth boundary containing no line segment. Then there exists 
a constant $c(K) > 0$  such that for any open, bounded $\Omega \subset \mathbb{R}^n$ and any decomposition
$$ \Omega = \bigcup_{i=1}^{N}{\Omega_i}$$
with measurable sets $\Omega_i$ satisfying
$$ \Omega_i \cap \Omega_j = \emptyset \qquad \mbox{for} \quad i \neq j$$
and $N$ sufficiently large, there is a geometric uncertainty principle
$$ \left(\sum_{i=1}^{N}{\frac{|\Omega_i|}{|\Omega|}\mathcal{A}_{K}(\Omega_i)}\right)+\left(\sum_{i=1}^{N}{\frac{|\Omega_i|}{|\Omega|}D(\Omega_i)}\right) \geq c(K).$$
\end{thm}
This is certainly not the most general form of the theorem. Let $\mathcal{S}$ be the set of bounded
sets in $\mathbb{R}^n$ such that $\mathbb{R}^n$ can be partitioned into translations and rotations of $\mathcal{S}$. 
Suppose $K$ is a bounded set satisfying
$$ \inf_{S \in \mathcal{S}}{\mathcal{A}_{S}(K)} > \varepsilon$$
for some $\varepsilon > 0$. Does this already imply a geometric uncertainty principle for $\mathcal{A}_K$ with a constant depending only on $\varepsilon$?

\section{Application to spectral problems}

\subsection{Introduction.}
Consider an open, bounded domain $\Omega \subset \mathbb{R}^2$. The Laplacian operator with Dirichlet conditions
gives rise to a sequence of eigenvalues $(\lambda_n)_{n \in \mathbb{N}}$ and associated eigenfunctions 
$(\phi_n)_{n \in \mathbb{N}}$, where
\begin{align*}
 -\Delta \phi_n &= \lambda_n \phi_n \quad \mbox{in}~\Omega \\
\phi_n &= 0 \qquad \mbox{on}~\partial\Omega.
\end{align*}
Laplacian eigenfunctions are of great intrinsic interest and have been extensively studied. One
natural question is to find bounds on the number of connected components of
$$ \Omega \setminus \left\{x \in \Omega: \phi_n(x) = 0\right\}.$$
Let us denote this quantity by $N(\phi_n)$. There are no nontrivial lower bounds on $N(\phi_n)$ in general. Denoting the smallest positive zero of the
Bessel function by $j \sim 2.40\dots$, the known upper bounds are as follows
\begin{align*}
N(\phi_n) &\leq n \quad  &&\mbox{(Courant, 1924)} \\
\limsup_{n \rightarrow \infty}{\frac{N(\phi_n)}{n}} &\leq  \left(\frac{2}{j}\right)^2    &&\mbox{(Pleijel, 1956)} \\
\limsup_{n \rightarrow \infty}{\frac{N(\phi_n)}{n}} &\leq  \left(\frac{2}{j}\right)^2 - 3 \cdot 10^{-9}    &&\mbox{(Bourgain, 2013)},
\end{align*}
where $(2/j)^2 \leq 7/10$. Polterovich \cite{polt} suggests that the optimal constant might be $2/\pi \sim 0.63$ with equality for a rectangle
(this example has also been noted B\'{e}rard \cite{ber} and probably others). It seems natural to assume that a domain $\Omega \subset \mathbb{R}^2$ giving rise to a 
large number of nodal domains needs to have a completely integrable geodesic flow. Some numerical experiments in
this direction have been carried out by Blum, Gnutzmann \& Smilansky \cite{blum}.

\subsection{Pleijel's argument.} Pleijel's argument \cite{plei} is short and simple. Suppose the eigenfunction $\phi_n$ induces a partition
$$ \Omega = \bigcup_{i=1}^{N}{\Omega_i}.$$
Then, by the Faber-Krahn inequality,
$$ \lambda_n(\Omega) \geq \lambda_1(\Omega_i) \geq \lambda_1(B),$$
where $B$ is the disk satisfying $|B| = |\Omega_i|$. However, $\lambda_1(B)$ can be explicitely computed and the inequality
then implies a lower bound on $|B|$. Combining this with Weyl's law $\lambda_n \sim 4\pi n/ |\Omega|$, yields the result.
Of course, this argument is only sharp if we have a decomposition of $\Omega$ into disks of equal radius.

\subsection{Bourgain's argument.} Bourgain \cite{bour} employs a spectral stability estimate due to Hansen \& Nadirashvili, which is 
formulated in terms of the inradius of a domain: for a nonempty, bounded domain $\Omega \subset \mathbb{R}^2$, we have
$$ \lambda_1(\Omega) \geq \left[1 + \frac{1}{250}\left(1-\frac{r_i(\Omega)}{r_o(\Omega)}\right)^3\right]\lambda_1(\Omega_0),$$
where $\Omega_0$ is the ball with $|\Omega_0| = |\Omega|$, $r_0(\Omega)$ is the radius of $\Omega_0$ and $r_i$ the
inradius of $\Omega$. The second ingredient is a packing result due to Blind \cite{blind}: the packing density of
a collection of disks in the plane with radii $a_1, a_2, \dots$ satisfying $a_i \geq (3/4) a_j$ for all $i,j$ is bounded 
from above by $\pi/\sqrt{12}.$ These two results imply the improvement.

\subsection{An improved Pleijel estimate.} Exploiting stability estimates for the Faber-Krahn inequality in terms of Fraenkel asymmetry,
we are able to prove the following result.

\begin{corollary}
 There exists a constant $\varepsilon_0 > 0$ such that
$$ \limsup_{n \rightarrow \infty}{\frac{N(\phi_n)}{n}} \leq \left(\frac{2}{j}\right)^2 - \varepsilon_0.$$
\end{corollary}

An explicit value for $\varepsilon_0$ would follow from an explicit constant in a Faber-Krahn stability result
involving Fraenkel asymmetry (these constants are known to exist but have not yet been determined explicitely). 
Given the general interest in this question, we are confident that such a result will be eventually 
obtained. Much like Bourgain, however, we consider the underlying geometry more interesting 
than the actual numerical value -- particularly in light of the following obstruction.

\subsection{An obstruction.}  Take $\Omega = [0,1]^2$ of unit measure and cover it using the hexagonal
covering (with obvious modifications at the boundary). Numerical computations (e.g. \cite{jia}) give that
the first Laplacian eigenvalue of a hexagon $H$ satisfies
$$ \lambda_1(H) \sim \frac{18.5762}{|H|}.$$
The Weyl law gives
$$ \lambda_n(\Omega) \sim 4\pi n.$$
We can place $N$ hexagons of size $|H|$ in $\Omega$, where
$$ N |H| = 1.$$
Since we need to have $\lambda_n(\Omega) \geq \lambda_1(H)$, this implies
$$ 4\pi n \sim \frac{18.5762}{|H|}$$
and thus
$$ N = \frac{1}{|H|} \sim \frac{4\pi}{18.5762} n \sim 0.676\dots n.$$
As a consequence, any type of argument that leads to an improved Pleijel inequality with a constant smaller than $0.67\dots$ will
need to employ completely different arguments: the arguments given by Pleijel, Bourgain and this paper argue based on
the \textit{assumption} that a partition of $\Omega$ into nodal domains is given. However, such a partition could
very well be the hexagonal partition. Arguments leading to a better constant than $0.676\dots$ will need 
to \textit{explain} why, say, an eigenfunction on a domain will not have eigenfunctions corresponding to a partition 
into hexagons.

\subsection{Spectral minimal partitions}
The problem of spectral minimal partitions is as follows: given a smooth, bounded domain $\Omega \subset \mathbb{R}^n$ and
an integer $k \in \mathbb{N}$, find among all partitions of $\Omega$ into $k$ disjoint domains
$$ \Omega = \bigcup_{i=1}^{k}{\Omega_i}$$
the one minimizing
$$ \max_{1 \leq i \leq k}{\lambda_1(\Omega_i)}.$$
It is conjectured that in two dimensions the minimal partitions should asymptotically behave like hexagonal tilings
(with the exception of the boundary, which becomes neglible as $k \rightarrow \infty$). We refer to Caffarelli \& Lin \cite{caf}, 
Helffer \& Hoffmann-Ostenhof \& Terracini \cite{heli} and a survey of Helffer \cite{helffer}. One basic inequality 
\cite[Proposition 6.1]{heli} following immediately from Pleijel's estimate is that
$$ \max_{1 \leq i \leq k}{\lambda_1(\Omega_i)} \geq k\frac{\pi j^2}{|\Omega|}.$$
Bourgain remarks that his argument also allows to slightly improve the constant in this inequality. As
a second quantity that is sometimes minimized (see e.g. Caffarelli \& Lin \cite{caf} or B\'{e}rard \& Helffer \cite{berr}),
one can consider the average and establish a strenghtened Pleijel-type estimate 
$$ \max_{1 \leq i \leq k}{\lambda_1(\Omega_i)} \geq \frac{1}{k}\sum_{i=1}^{k}{\lambda_1(\Omega_i)} \geq  k\frac{\pi j^2}{|\Omega|}.$$
This inequality, too, can be strenghtened.
\begin{corollary}
There exists a $\varepsilon_0 > 0$ such that for any smooth, bounded domain $\Omega \subset \mathbb{R}^2$ and all $k$ sufficiently large 
(depending on $\Omega$)
$$\frac{1}{k}\sum_{i=1}^{k}{\lambda_1(\Omega_i)} \geq  (\pi j^2+\varepsilon_0)\frac{k}{|\Omega|}.$$
\end{corollary}

\section{Proof of Theorem 1 in two dimensions}
This section contains a complete proof of the main statement in dimension $n=2$: the proof will track all arising constants. This takes up most 
of the text and contains all the ideas of this paper -- the argument is robust and the necessary (and rather easy) modifications to obtain the 
more general results will then be given in subsequent sections.

\subsection{Two possible strategies.} There seems to be a very natural way to prove the statement, however,
we did not manage to fully quantify all the steps and had to find another argument. We record our original 
idea nonetheless in the hope of building additional insight.\\

\textit{Sketch of an idea.} The inequality can be regarded as a probabilistic statement. Pick a random domain weighted 
according to size (i.e. choosing a random point of the domain, the probability of picking $\Omega_i$ is $|\Omega_i|/|\Omega|$). 
Our statement can be read as a lower bound on the expectation of the random variable
$$ \mathcal{A}(\Omega_i) + D(\Omega_i).$$
This motivates the following argument. Pick a random domain: either it already
has large Fraenkel asymmetry (in which case we are done) or it does not and behaves quite disk-like. In the second case,
we look at its neighbouring domains. If there are few adjacent domains, at least one of them touches along a
long arc of the boundary meaning that the neighbouring domain has large Fraenkel asymmetry (two disks touch in at
most one point). If there are many neighbours, either most are significantly smaller (making our randomly chosen domain
big in comparison and giving the statement) or some will need to get squeezed together because there is not enough room (creating
a large Fraenkel asymmetry). We believe that such a strategy, properly implemented, could give a relatively sharp constant -- 
however, making all these steps quantitative seems complicated.\\

\textit{Sketch of a different idea: our proof.} We chose a different approach of a more global nature: given a decomposition, we immediately switch to a collection
of $N$ disks by taking disks realizing the Fraenkel asymmetry for each partition. Then, we show that
\begin{itemize}
 \item there are few very large elements: the size of neighbourhood of the union of all disks whose size is bounded away from the smallest element in the partition
by a constant factor can be bounded from above.
\item ignoring the large sets (of which there are few), the Fraenkel balls of small sets usually do not overlap
too much; the exceptional set is small.
\end{itemize}
Removing all large disks and all overlapping disks, we may shrink the remaining disks such that no two of them overlap:
the resulting disk packing cannot have too high a density.

\subsection{Defining quantities.} The limes inferior in the statement guarantees that boundary effects coming
from $\partial \Omega$ become neglible and we will ignore the boundary throughout the proof (equivalently, we
could have phrased the statement for periodic partitions of $\mathbb{R}^2$).\\

We assume w.l.o.g. that $|\Omega| = 1$. For a point $x \in \mathbb{R}^2$ and a set $A \subset \mathbb{R}^2$,
we abbreviate
$$ \left\|x-A\right\| := \inf_{y \in A}{\|x-y\|}.$$
We introduce two numbers $c_1, c_2 > 0$ that will serve as threshold values for 'being big' and
'strong overlap' and we will keep them as variables throughout the proof, however, a minimization problem
towards the end of the proof will motivate us to set
$$c_1 = \frac{1}{250} \qquad \mbox{and} \qquad c_2 = \frac{7}{250}$$
and the reader can substitute these values throughout the proof if he wishes to. Their role is as follows: we call $\Omega_i$ 'big', if
\begin{align} \label{eq:c1}
 |\Omega_i| \geq (1+c_1)\min_{1 \leq j \leq N}{|\Omega_j|}.
\end{align}
The constant $c_2$ will serve as a measure of overlap: two disks with centers in $x, y \in \mathbb{R}^2$
and radii $r_1, r_2$ will be considered to have 'large' overlap if
\begin{equation} \label{eq:c2}
 |x - y| \leq (1-c_2)(r_1+r_2).
\end{equation}
We define a natural length scale $\eta_0$. \textit{Everything} in this problem and this proof is scale-invariant and, correspondingly, the actual size 
of $\eta_0$ is completely irrelevant throughout the proof: the variable cancels in the end. However, we consider it helpful to imagine a fixed length 
scale $\eta_0$ at which everything plays out and will phrase all arising quantities in terms of $\eta_0$, which we define via
\begin{equation} \label{eq:eta} \pi \eta_0^2 = \min_{1 \leq i \leq N}{|\Omega_i|}.\end{equation}
 The proof will be carried out via contradiction, we assume
$$ \left(\sum_{i=1}^{N}{\frac{|\Omega_i|}{|\Omega|}\mathcal{A}(\Omega_i)}\right)+\left(\sum_{i=1}^{N}{\frac{|\Omega_i|}{|\Omega|}D(\Omega_i)}\right) \leq c$$
for some small constant $c$ and show that this will lead to a contradiction if $c$ is small enough. It makes sense
to be slightly more careful and so we assume that for all $d_1, d_2 \geq 0$ with $c = d_1 + d_2$
\begin{equation}  \label{eq:assum1}
\left(\sum_{i=1}^{N}{\frac{|\Omega_i|}{|\Omega|}\mathcal{A}(\Omega_i)}\right) \leq d_1 
\end{equation}
and
\begin{equation} \label{eq:assum2}
\left(\sum_{i=1}^{N}{\frac{|\Omega_i|}{|\Omega|}\mathcal{D}(\Omega_i)}\right) \leq d_2. 
\end{equation}

We assign to each of the sets $\Omega_1, \dots, \Omega_n$ a disk $B_1, B_2, \dots, B_n$ such that $|B_i| = |\Omega_i|$ and
$$ \mathcal{A}(\Omega_i) = \frac{| \Omega_i \triangle B_i|}{|\Omega_i|}.$$
Note that a disk $B_i$ need not be uniquely determined by $\Omega_i$ (if there is more than one
 possible choice, we pick an arbitrary one and fix it for the rest of the proof). Each of these disks $B_i$
has a center $x_i$ and a radius $r_i \geq \eta_0$.\\

\subsection{The union of large sets has small measure.} Here we prove a simple statement: the measure
occupied by 'large' sets (in the sense of \eqref{eq:c1}) is small. Note that the statement is indeed
for the measure and not the number of large sets, which could be small. 

\begin{lemma} We have 
\begin{equation} \label{eq:bound1}
  \left|  \bigcup_{|\Omega_i| > (1+c_1)\pi \eta_0^2}^{}{\Omega_i} \right| \leq \frac{d_2}{c_1}+d_2.
\end{equation}
\end{lemma}
\begin{proof}
 From  \eqref{eq:eta}, \eqref{eq:assum2} and the definition of $D(\Omega_i)$, we get that
$$ d_2 \geq \sum_{i=1}^{N}{\frac{|\Omega_i|}{|\Omega|}D(\Omega_i)} = \sum_{i=1}^{N}{(|\Omega_i| - \pi\eta_0^2)} = 1 - N\pi\eta_0^2$$
and therefore
$$ N \geq \frac{1-d_2}{\pi \eta_0^2}.$$ 
Now, let us suppose that $0 \leq M \leq N$ elements of the partition are 'small' in the sense of satisfying $|\Omega_i| \leq (1+c_1)\pi\eta_0^2$. We wish to show that
$M$ itself has to be big. Trivially,
$$ \left| \bigcup_{|\Omega_i| \leq (1+c_1)\pi \eta_0^2}^{}{\Omega_i} \right| \geq M\pi \eta_0^2.$$

The remaining measure is divided among big sets, hence the number of 'big' elements is at most the remaining
measure divided by the smallest possible area a 'big' set can have
$$ N - M \leq \frac{1-M\pi\eta_0^2}{(1+c_1)\pi\eta_0^2}$$
and thus, in total,
$$ \frac{1-d_2}{\pi \eta_0^2} \leq N = M + (N-M) \leq M +\frac{1-M\pi\eta_0^2}{(1+c_1)\pi\eta_0^2}.$$
Rewriting gives
$$ M \geq \frac{c_1-d_2-d_2c_1}{c_1 \pi \eta_0^2},$$
which implies 
$$ \left|  \bigcup_{|\Omega_i| \leq (1+c_1)\pi \eta_0^2}^{}{\Omega_i} \right| \geq \frac{c_1-d_2-d_2c_1}{c_1}$$
and therefore, since $|\Omega| = 1$,
$$ \left|  \bigcup_{|\Omega_i| > (1+c_1)\pi \eta_0^2}^{}{\Omega_i} \right| \leq \frac{d_2}{c_1}+d_2.$$
\end{proof}

\subsection{A neighbourhood of the union of large sets has small measure.} In the last section we have
seen that measure of the set of large disks is small. However, we actually require a slightly stronger
statement showing that an entire neighbourhood of that set is still small. For future use, we define the 
index set $I$ of partition elements with 'big' measure
$$ I = \left\{i \in \left\{1, \dots, N\right\}: |\Omega_i| \geq (1+c_1)\pi \eta_0^2\right\}.$$
\begin{lemma} A $2\eta_0-$neighbourhood of $ \bigcup_{i \in I}{B_i}$ has small measure: we have
\begin{equation} \label{eq:neigh1}
 \left| \left\{x \in \Omega: \left\| x -  \bigcup_{i \in I}{B_i}\right\| \leq 2\eta_0\right\}\right| \leq \frac{9d_2}{c_1} + 9d_2.
\end{equation}
\end{lemma}
\begin{proof}
 This argument is very simple: the $2\eta_0-$neighbourhood of a disk with radius $r$ has measure $(r+2\eta_0)^2\pi$. The worst case
is precisely the case, where all $B_i$ are well-separated such that their $2\eta_0-$neighbourhoods do not intersect (otherwise:
move the disks apart to create a neighbourhood with bigger measure). In this case, the total measure gets amplified by factor
$$ \frac{(\sqrt{1+c_1}+2)^2\eta_0^2 \pi}{(1+c_1)\eta_0^2 \pi} \leq 9$$
and the result follows from \eqref{eq:bound1}.
\end{proof}

\subsection{Most small sets have well-separated balls.} By now, we have a good control on the 'large' disks and their neighbourhood:
we can (mentally and later in the proof literally) remove them from the stage and consider the remaining small disks. It remains to 
control their intersections.

\begin{lemma} The union of 'small' disks $B_i$, $i \neq I$, for which there exists another 'small' disk such that they intersect
strongly in the sense of \eqref{eq:c2} is bounded by
 \begin{equation} \label{eq:bound2} 
 \left| \bigcup_{i \notin I}^{}{\left\{B_i: \exists_{j \notin I} i \neq j:~ |x_i-x_j| \leq (1-c_2)(r_i+r_j)\right\}} \right| \leq \frac{20\pi}{37}\frac{1+c_1}{c_2^{3/2}}d_1
\end{equation}
\end{lemma}

\begin{proof}
 
 For simplicity, we introduce the index set
$$ J = \left\{i \notin I: \exists B_i~~ \exists_{j \notin I} ~i \neq j:~ |x_i-x_j| \leq (1-c_2)(r_i+r_j)\right\}.$$
We will now derive an upper bound on the measure of the set, which we now can abbreviate as $\cup_{j \in J}{B_j},$
using nothing but the inequality \eqref{eq:assum2}
$$  \sum_{i=1}^N{|\Omega_i|\mathcal{A}(\Omega_{i})} \leq d_1.$$

Suppose $i \in J$. Then there exists a $j \in J$ such that the balls $B_i, B_j$ have controlled radius (this follows
automatically from the fact that both disks are 'small')
$$ \eta_0 \leq r_i, r_j \leq \sqrt{1+c_1}\eta_0$$
and intersect in a quantitatively controlled way
$$ |x_i-x_j| \leq (1-c_2)(r_i+r_j).$$
Then the intersection $B_i \cap B_j$ is of interest: if the Fraenkel asymmetry of $\Omega_i$ is to be small, then almost
all of its measure should be contained in $B_i$ but the very same reasoning also holds for $\Omega_j$ and $B_j$. In particular,
since every point in the intersection can only belong to one of the two sets, we have 
\begin{equation} \label{eq:diamond}
|\Omega_i|\mathcal{A}(\Omega_i) + |\Omega_j|\mathcal{A}(\Omega_j) \geq |B_i \cap B_j|.
\end{equation}
It remains to compute the quantity $|B_i \cap B_j|$. Using scaling invariance, we may assume $\eta_0 = 1$. We are then dealing with
two disks in the Euclidean plane whose radii $r_1, r_2$ are bounded from below by 1 and whose centers $x_1, x_2$ satisfy
$$ d := |x_1-x_2| \leq (1-c_2)(r_1+r_2).$$
Elementary Euclidean geometry yields 
\begin{align*}
 |B_i \cap B_j| &= r_1^2 \arccos{\left(\frac{d^2 + r_1^2 - r_2^2}{2dr_1}\right)}  
+  r_2^2 \arccos{\left(\frac{d^2 + r_2^2 - r_1^2}{2dr_2}\right)} \\
&-\frac{1}{2}\sqrt{(-d+r_1+r_2)(d+r_1-r_2)(d-r_1+r_2)(d+r_1+r_2)}.
\end{align*}
Easy but tedious calculation give that the quantity is decreasing in both radii and as such minimized for $r_1 = r_2 = 1$. This
is then a one-dimensional function in $c_2$ and it is easy to show that for $c_2 \leq 0.05$ the function is
$$2\arccos{\left(1-c_2\right)} - 2\sqrt{(2-c_2)(1-c_2)^2c_2} \geq \frac{37}{10}c_2^{\frac{3}{2}}.$$
Recalling the normalization $\eta_0 = 1$, we get the scale-invariant estimate
 $$|B_i \cap B_j| \geq \frac{37}{10}c_2^{\frac{3}{2}}\eta_0.$$

A priori, the intersection patterns of $\left\{B_i: i \in J\right\}$ can be very complicated. However, there is a very
simple monotonicity: we can remove areas, where three or more balls intersect and arrange the balls in (possibly more
than one) chain. This increases the area and decreases the area of intersection.
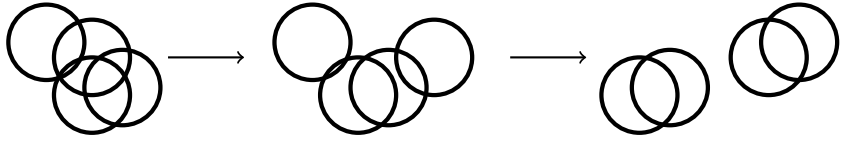
\begin{figure}[h!]
\begin{tikzpicture}
\draw [ultra thick] (0,0) circle [radius=0.5];
\draw [ultra thick] (0.4,0.1) circle [radius=0.5];
\draw [ultra thick] (-0.6,0.7) circle [radius=0.5];
\draw [ultra thick] (0,0.5) circle [radius=0.5];
\draw [->, thick] (1,0.5) -- (2,0.5);
\draw [ultra thick] (3.5,0) circle [radius=0.5];
\draw [ultra thick] (3.9,0.1) circle [radius=0.5];
\draw [ultra thick] (2.9,0.7) circle [radius=0.5];
\draw [ultra thick] (4.5,0.5) circle [radius=0.5];
\draw [->, thick] (5.5,0.5) -- (6.5,0.5);
\draw [ultra thick] (7.2,0) circle [radius=0.5];
\draw [ultra thick] (7.6,0.1) circle [radius=0.5];
\draw [ultra thick] (9.3,0.7) circle [radius=0.5];
\draw [ultra thick] (8.9,0.5) circle [radius=0.5];
\end{tikzpicture}
\caption{Increasing area while decreasing average Fraenkel asymmetry}
\end{figure}
By the same argument, the area further increases if we assume that any disk in $\left\{B_i: i \in J\right\}$
touches precisely one other disk (i.e. the intersection pattern reduces to that of pairs of disks intersecting
each other and no disk). Any such (i.e. intersecting) pair of disks $B_i, B_j$ satisfies
$$ |B_i \cup B_j| \leq 2(1+c_1)\pi \eta_0^2$$
as well as
$$ |B_i \cap B_j| \geq \frac{37}{10}c_2^{\frac{3}{2}}\eta_0,$$
which is connected via \eqref{eq:diamond} to the sum 
$$  \sum_{i=1}^N{|\Omega_i|\mathcal{A}(\Omega_{i})} \leq d_1.$$
Thus
$$  \left| \bigcup_{j \in J}^{}{B_j} \right|
\leq \left[ 2(1+c_1)\pi \eta_0^2 \right]\frac{d_1}{\frac{37}{10}c_2^{3/2}\eta_0^2} = \frac{20\pi}{37}\frac{1+c_1}{c_2^{3/2}}d_1.$$
\end{proof}

\subsection{Bounds on the size of the neighbourhood of strongly intersecting small disks.}
By applying the very same reasoning as in Section 3.4, we could argue that by considering an entire $2\eta_0$-neighbourhood the measure
gets amplified by a factor of at most 9. This is perfectly reasonable but can actually be improved as we are now dealing with disks
intersecting other disks. We are thus studying the following problem: given two disks $B_1, B_2$ with radii $r_1, r_2 \geq \eta_0$
intersecting in precisely one point, what bounds can be proven on
$$ \frac{ |\left\{x \in \mathbb{R}^2: \|x-(B_1 \cup B_2)\| \leq 2\eta_0\right\}|}{|B_1| + |B_2|} \leq ?$$
This problem can be explicitely solved using elementary calculus and reduces to a case-distinction and two
integrations; we leave the details to the interested reader. Carrying out the calculations gives
$$ \frac{ |\left\{x \in \mathbb{R}^2: \|x-(B_1 \cup B_2)\| \leq 2\eta_0\right\}|}{|B_1| + |B_2|} \leq \frac{9}{2} + \frac{2\sqrt{2}}{\pi} + \frac{9}{\pi}\arcsin{\left(\frac{1}{3}\right)} \sim 6.37\dots \leq \frac{32}{5}$$
with equality for $r_1 = r_2 = 1$. Arguing as in Section 3.4 and using \eqref{eq:bound2}, we get 
\begin{equation} \label{eq:neigh2}
 \left| \left\{x \in \Omega: \left\| x -  \bigcup_{j \in J}^{}{B_j} \right\| \leq 2\eta_0\right\}\right| 
\leq   \frac{128\pi}{37}\frac{(1+c_1)}{c_2^{3/2}}d_1
\end{equation}

\subsection{Finding a dense disk packing.} We conclude our argument by deriving the existence of a disk
packing in the plane with impossible properties. Here, we employ an aforementioned result of Blind \cite{blind} 
that also played a role in Bourgain's argument and was mentioned before: the packing density of a collection of disks in the 
plane with radii $a_1, a_2, \dots$ satisfying $a_i \geq (3/4) a_j$ for all $i,j$ is bounded from above by $\pi/\sqrt{12}.$\\

A rough outline of the remainder of the argument is as follows
\begin{enumerate}
 \item consider the set of Fraenkel disks $\left\{B_i: 1\leq i \leq N\right\}$
 \item remove all 'big' disks
 \item remove all remaining 'small' disks strongly intersecting another small disk
 \item shrink all remaining disks by a factor $(1-c_2)$.
\end{enumerate}
This leaves us with a set of \textit{disjoint} disks in the Euclidean plane with roughly the same radius and we
can apply Blind's result -- the argument has one big flaw, of course, removing elements from a set does \textit{not}
increase its packing density (just think of a hexagonal packing of disks: if we remove the little triangle-shaped
gaps between the disks, packing density goes up to 1). \\

We counter the problem by not only removing 'big' disks or 'small' disks strongly intersecting other small
disks but an entire $2\eta_0-$neighbourhood of these sets as well. Doing this is equivalent to assuming that while
we created holes in the middle of the set, these holes are of such a shape that within a neighbourhood we can
actually achieve packing density 1.\\

From \eqref{eq:neigh1} and \eqref{eq:neigh2}, we get that the set
$$ \Omega^* := \Omega \setminus \left( \left\{x \in \Omega: \left\| x -  \bigcup_{i \in I}{B_i}\right\| \leq 2\eta_0\right\}  \cup 
\left\{x \in \Omega: \left\| x -  \bigcup_{j \in J}^{}{B_j} \right\| \leq 2\eta_0\right\}\right)$$
satisfies
$$ |\Omega^*| \geq 1- \left( \frac{9d_2}{c_1} + 9d_2 + \frac{128\pi}{37}\frac{(1+c_1)}{c_2^{3/2}}d_1 \right).$$
$\Omega^*$ consists of disks with radii satisfying
$$ \eta_0 \leq r_i \leq \sqrt{1+c_1}\eta_0$$
and with the additional property that the centers of any two disks are well-seperated
$$|x_i-x_j| \geq (1-c_2)(r_i+r_j).$$
By shrinking all these disks by a factor of $1-c_2$ while keeping their center in the same place, they become disjoint. Thus, from Blind's result
$$ \left|(1-c_2)\Omega^*\right| \leq (1-c_2)^2\left[1- \left( \frac{9d_2}{c_1} + 9d_2 + \frac{128\pi}{37}\frac{(1+c_1)}{c_2^{3/2}}d_1 \right)\right] \leq \frac{\pi}{\sqrt{12}}.$$ 
We need to find a set of parameters, for which the inequality fails. Indeed, setting
$$c_1 = \frac{1}{250} \qquad \mbox{and} \qquad c_2 = \frac{7}{250},$$
we get for any $d_1, d_2 \geq 0$ with
$$ d_1 + d_2 = \frac{1}{60000},$$
that
$$(1-c_2)^2\left[1- \left( \frac{9d_2}{c_1} + 9d_2 + \frac{128\pi}{37}\frac{(1+c_1)}{c_2^{3/2}}d_1 \right)\right] \geq \frac{\pi}{\sqrt{12}} + \frac{1}{1000}.$$
This contradiction proves the statement. $\qed$\\

\textbf{Remark.} The weakest point in the argument is certainly the last step, where we remove an entire $2\eta_0-$neighhbourhood. Intuition
suggests that we should be able that maybe even removing merely a $\eta_0$-neighbourhood should be more than sufficient, however, we have
not been able to make progress on that question, which would certainly be the most natural starting point if one wanted to improve the constant
using arguments along these lines.

\section{Proof of the general case.} Here we give a proof of Theorem 2 in general dimensions (which contains Theorem 1 for $n \geq 3$ as a special case). 
This section essentially recapitulates the previous argument without caring about the actual numerical values at all. The new ingredient is the following insight:
in the proof of Theorem 1, after a careful geometric analysis, we did end up with the inequality
$$(1-c_2)^2\left[1- \left( \frac{9d_2}{c_1} + 9d_2 + \frac{128\pi}{37}\frac{(1+c_1)}{c_2^{3/2}}d_1 \right)\right] \geq \frac{\pi}{\sqrt{12}} + \frac{1}{1000}.$$
The crucial point is the following: no matter what actual numerical values are placed in front, by choosing $d_2 \ll c_1$ and $d_1 \ll c_2$, the inequality will
always be false for $c_1, c_2$ sufficiently close to 1 by simple continuity. In the previous proof, it was our goal to keep $d_1, d_2$ as large as possible but once we discard this concern, we can much more
wasteful in the actual geometric estimates.

\begin{proof} The argument is again by contradiction. $\eta_0$ plays a similar same role as before, we define it via
$$ \eta_0 = \left(\min_{1 \leq j \leq N}{|\Omega_j|}\right)^{1/n}.$$
The constant $c_1$ again determines whether a domain is 'big', which we define to be the case if
$$ |\Omega_i| \geq (1+c_1) \min_{1 \leq j \leq N}{|\Omega_j|}.$$
The precise meaning of $c_2$ is introduced further below. Arguing by contradiction we assume that
$$ \left(\sum_{i=1}^{N}{\frac{|\Omega_i|}{|\Omega|}\mathcal{A}_{K}(\Omega_i)}\right)+\left(\sum_{i=1}^{N}{\frac{|\Omega_i|}{|\Omega|}D(\Omega_i)}\right) \leq c$$
and want to derive a contradiction for $c$ sufficiently small.
Following the same argument as before, we again get a bound on the number of large sets
$$ \left|  \bigcup_{|\Omega_i| >(1+c_1)\eta_0^n}^{}{\Omega_i} \right| \leq \frac{c}{c_1}+c.$$
Switching again to the Fraenkel bodies $K_1, \dots, K_N$, we wish to remove a $c_3 \eta_0$ neighbourhood
of any 'large' Fraenkel body $K_i$, where $c_3 < \infty$ is chosen such that $c_3 \eta_0$ is many multiples
of the diameter of a 'small' $K_i$ having measure at most $(1+c_1)\eta_0^n$. This allows us to bound the size of a $c_3 \eta_0$
neighbourhood of 
 $$\bigcup_{|\Omega_i| >(1+c_1)\eta_0^n}^{}{K_i}$$
by $c_4(c/c_1+c)$ for some finite constant $c_4$. The constant $c_2$ now measures whether two 'small' Fraenkel
bodies have large intersection, writing again
$$ I = \left\{i \in \left\{1, \dots, N\right\}: |\Omega_i| \geq (1+c_1)\eta_0^n\right\},$$
we consider
\begin{align*}
\bigcup_{i \notin I}^{}{\left\{K_i: \exists_j i \neq j \notin I:~ |(K_i \cap K_j)| \geq c_2 \eta_0^n\right\}}.
\end{align*}
The same argument as before implies that for any two elements in the set, we get
$$ \mathcal{A}_{K}(K_i)|K_i| + \mathcal{A}_{K}(K_j)|K_j| \geq c_2 \eta_0^n.$$
Since 
$$  \left(\sum_{i=1}^{N}{\frac{|\Omega_i|}{|\Omega|}\mathcal{A}_{K}(\Omega_i)}\right) \leq c,$$
this implies a bound on the measure of the set 
$$  \left| \bigcup_{i \notin I}^{}{\left\{K_i: \exists_j i \neq j \notin I:~ |(K_i \cap K_j)| \geq c_2 \eta_0^n\right\}} \right| \leq c_5 c$$
for some constant $c_5 < \infty$ and a bound of the form $c_6c$ on the measure of its $c_3 \eta_0$ neighbourhood. Finally, since the boundary of the convex body $K$
contains no line segment, we get that for every $\varepsilon_1 > 0$ there is a $\varepsilon_2 > 0$ such 
that any collection $K_1, K_2, \dots$ of nonoverlapping rotated and scaled translates of $K$ in the plane with volumes $v_1, v_2, \dots$ satisfying 
$$ \inf_{i,j}{\frac{v_i}{v_j}} \geq 1-\varepsilon_1$$
has packing density at most $1-\varepsilon_2$. Finally, there exists a constant $c_7$ such that for any two scaled, translated copies $K_1, K_2$ of $K$ with
$$|(K_i \cap K_j)| \leq c_2 \eta_0^n,$$
the rescaled bodies $c_7 K_1, c_7K_2$ (rescaling being done in a way to fix, say, their center of mass) satisfy
$$ (c_7 K_1) \cap (c_7 K_2) = \emptyset.$$
Note that the optimal $c_7$ depends continuously on $c_2$ and tends to 1 as $c_2$ tends to 0. Now, following the same argument as before, we can derive the inequality
$$ 1- \varepsilon_2 \geq c_7^n\left(1-\frac{c_4 c}{c_1} - c_4 c - c_6 c\right).$$
The dependence is easy: pick some $0 < \varepsilon_1 \ll 1$. This yields $\varepsilon_2 > 0$. Given $\varepsilon_1$, pick $c_1 \ll \varepsilon_1$. We pick
$c_2$ so small that $c_7^n > 1-\varepsilon_2$. $c_4$ and $c_6$ are again externally given but the inequality can now be seen to be false
if $c=0$. By continuity $c > 0$.
\end{proof}

\subsection{Proof of the improved Pleijel estimate.} 
The Corollary has a very simple proof: as in the proof of Pleijel's estimate, we get a lower bound on
$$ \min_{1 \leq i \leq N}{|\Omega_i|}$$
from the Faber-Krahn inequality. Theorem 1 now implies that either not all elements in the
partition are of that size (in which case some need to be bigger and their requirement for more spaces
allows for a smaller number of nodal domains) or that some deviate from the disk in a controlled way
(in which case stability estimates require them to have a larger measure).

\begin{proof} Let $$ \Omega = \bigcup_{i=1}^{N}{\Omega_i}$$
be the decomposition introduced by a Laplacian eigenfunction with eigenvalue $\lambda \gg 1$ and let $\eta_0 = \eta_0(\lambda)$ be chosen
in such a way that $\pi \eta_0^2 = |B|,$ where $B$ is the disk such that $\lambda_1(B) = \lambda$. 
Theorem 1 yields that
$$\sum_{i=1}^N{\frac{|\Omega_i|}{|\Omega|}\left(\mathcal{A}(\Omega_{i}) + D(\Omega_{i})\right) } \geq c$$
for some $c \geq 1/60000$; therefore, either
$$ \sum_{i=1}^{N}{\frac{|\Omega_i|}{|\Omega|}D(\Omega_{i})} \geq \frac{c}{2} \qquad \mbox{or} \qquad \sum_{i=1}^{N}{\frac{|\Omega_i|}{|\Omega|}\mathcal{A}(\Omega_{i})} \geq \frac{c}{2}.$$
Suppose the first inequality holds. Then
$$ \frac{c}{2} \leq  \sum_{i=1}^{N}{\frac{|\Omega_i|}{|\Omega|}D(\Omega_{i})} = \frac{1}{|\Omega|}\left(|\Omega| - N\pi \eta_0^2\right)$$
in which case
$$ N \leq \left(1-\frac{c}{2}\right)\frac{|\Omega|}{\pi \eta_0^2}.$$
The fact that Pleijel's argument is sharp for a partition into equally sized disks (or, equivalently, Weyl's law) implies
$$ \lim_{\lambda \rightarrow \infty}{\frac{|\Omega|}{\pi \eta_0^2}} = \left(\frac{2}{j}\right)^2n$$
and this yields the result. Suppose the second inequality holds. We start with a simple Lemma.
\begin{lemma}
 We have
\begin{equation} \label{eq:lem}
 \left| \bigcup_{\mathcal{A}(\Omega_i) \geq \frac{c}{6}}{ \Omega_i} \right| \geq \frac{c}{6} |\Omega|.
\end{equation}

\end{lemma}
\textit{Proof of the Lemma.} Suppose the statement was false. Then, using $\mathcal{A}(\Omega_i) \leq 2$,
\begin{align*}
 \frac{c}{2} &\leq \sum_{i=1}^{N}{\frac{|\Omega_i|}{|\Omega|}\mathcal{A}(\Omega_{i})}  \\
&< \frac{2}{|\Omega|} \left| \bigcup_{\mathcal{A}(\Omega_i) \geq  \frac{c}{6}}{ \Omega_i} \right| + 
\frac{c}{6}\frac{1}{|\Omega|} \left| \bigcup_{\mathcal{A}(\Omega_i) \leq  \frac{c}{6}}{ \Omega_i} \right| \\
&\leq \frac{c}{3} + \frac{c}{6} = \frac{c}{2}. \qedhere
\end{align*}

Now we recall some stability estimates 
for the Faber-Krahn inequality in terms of Fraenkel asymmetry. 
Brasco, De Philippis \& Velichkov \cite{bra} (improving an earlier result of Fusco, Maggi \& Pratelli \cite{fus}) 
have shown that
$$ \frac{\lambda_1(\Omega) - \lambda_1(\Omega_0)}{\lambda_1(\Omega_0)} \gtrsim  \mathcal{A}(\Omega)^2,$$
where $\Omega_0$ is again the disk with $|\Omega_0| = |\Omega|$.\\

Pick any domain $\Omega_i$ with $\mathcal{A}(\Omega_i) \geq c/6$ and use $B$ to denote the disk such that $|B| = |\Omega_i|$. The stability estimate
$$ \frac{\lambda_1(\Omega_i) - \lambda_1(B)}{\lambda_1(B)} \geq C \cdot \mathcal{A}(\Omega_i)^2$$
 can be rewritten as
\begin{equation} \label{eq:ok} \lambda_1(\Omega_i) \geq \left(1+C\frac{c^2}{36}\right)\lambda_1(B)\end{equation}
Recall that $\eta_0$ is chosen such that the disk $D$ of radius $\eta_0$ satisfies $\lambda_1(D) = \lambda_n(\Omega)$. However,
by \eqref{eq:ok}, we know that $\lambda_1(\Omega_i)$ is a multiplicative factor larger than the first eigenvalue of the disk
of equal measure. In order for $\lambda_1(\Omega_i) \leq \lambda_n(\Omega)$ to still be satisfied, we require that
\begin{equation}
 \label{eq:plei}
\frac{|\Omega_i|}{\pi \eta_0^2} \geq 1+C\frac{c^2}{36}.
\end{equation}

We use this as follows:
$$\Omega = \bigcup_{i=1}^{N}{\Omega_i} = \left(\bigcup_{\mathcal{A}(\Omega_i) \geq  \frac{c}{6}}{\Omega_i}\right)
\cup \left(\bigcup_{\mathcal{A}(\Omega_i) \leq  \frac{c}{6}}{\Omega_i}\right)$$
By Pleijel's argument, the number of nodal domains in the second set is bounded from above by
$$ \frac{1}{\pi \eta_0^2}\left|\bigcup_{\mathcal{A}(\Omega_i) \leq  \frac{c}{6}}{\Omega_i}\right|$$
while \eqref{eq:plei} implies the number of nodal domains in the first set is bounded by
$$ \frac{1}{\pi \eta_0^2}\frac{1}{1+C\frac{c^2}{36}}\left|\bigcup_{\mathcal{A}(\Omega_i) \geq  \frac{c}{6}}{\Omega_i}\right|.$$
Using \eqref{eq:lem} and $|\Omega| = 1$, we get
\begin{align*}
 \frac{1}{\pi \eta_0^2}\frac{1}{1+C\frac{c^2}{36}}\left|\bigcup_{\mathcal{A}(\Omega_i) \geq  \frac{c}{6}}{\Omega_i}\right| +
\frac{1}{\pi \eta_0^2}\left|\bigcup_{\mathcal{A}(\Omega_i) \leq  \frac{c}{6}}{\Omega_i}\right| &\leq 
\frac{1}{\pi \eta_0^2}\frac{1}{1+C\frac{c^2}{36}}\frac{c}{6} + \frac{1}{\pi \eta_0^2}\left(1-\frac{c}{6}\right) \\
&\leq  \left(1-\frac{c^3 C}{216 + 6 c^2 C}\right)\frac{1}{\pi \eta_0^2}.
\end{align*}
By definition of $\eta_0$, the expression $(\pi \eta_0^2)^{-1}$ is precisely the upper bound of Pleijel on the number of
nodal domains and thus
\begin{align*}
 N \leq \left(1-\frac{c^3 C}{216 + 6 c^2 C}\right)\left(\frac{2}{j}\right)^2 n
\end{align*}
for $n$ sufficiently big.
\end{proof}

\subsection{Proof of the Spectral Partition Inequality.} 
\begin{proof} Suppose the statement was false. Then there exists some smooth, bounded $\Omega$ such that for any $\varepsilon > 0$ there are
arbitrarily large $k$ such that there are partitions
$$ \Omega = \bigcup_{i=1}^{k}{\Omega_i}$$
with
$$\frac{1}{k}\sum_{i=1}^{k}{\lambda_1(\Omega_i)} \leq  (\pi j^2+\varepsilon)\frac{k}{|\Omega|}.$$
Let us start by re-iterating the proof of the original estimate. The Faber-Krahn inequality implies 
$$ \frac{1}{k}\sum_{i=1}^{k}{\lambda_1(\Omega_i)} \geq \frac{1}{k}\sum_{i=1}^{k}{\frac{\pi j^2}{|\Omega_i|}}.$$
The convexity of $x \rightarrow 1/x$ and the fact that
$$ \sum_{i=1}^{k}{|\Omega_i|} = |\Omega|$$
immediately imply that
$$\frac{1}{k}\sum_{i=1}^{k}{\frac{\pi j^2}{|\Omega_i|}} \geq k\frac{\pi j^2}{|\Omega|}$$
with equality if and only if $|\Omega_i| = |\Omega|/k$ for all $1 \leq i \leq k$. We can quantify the
notion of convexity a little bit. Indeed, assuming our desired spectral partition inequality to be false and given
any $\delta > 0$, we can find a subsequence of partitions with
$$ \frac{1}{k_j}\#\left\{1 \leq i \leq k_j: (1-\delta)|\Omega| \leq  k_j|\Omega_i| \leq (1+\delta)|\Omega|\right\} \rightarrow 1.$$
This, however, means that we can find a subsequence of partitions with the propery that
$$\left(\sum_{i=1}^{k_j}{\frac{|\Omega_i|}{|\Omega|}D(\Omega_i)}\right) \leq \delta',$$
where $\delta' > 0$ can be as small as we wish. Then, however, the geometric uncertainty principle implies that
$$\left(\sum_{i=1}^{k_j}{\frac{|\Omega_i|}{|\Omega|}\mathcal{A}(\Omega_i)}\right) \geq c_2 - \delta',$$
where $c_2 > 1/60000$ is the optimal constant in two dimensions. Then, however, arguing as before, we
can improve on Pleijel's estimate.
\end{proof}

\textbf{Acknowledgments.} The author is grateful for various discussions with Bernhard Helffer, who taught him about
spectral partition problems; this interaction took place at a workshop in the Banff International Research Station
which is to be thanked for its hospitality. Suggestions from an anonymous referee greatly increased the quality of
exposition.


\begin{thebibliography}{5}

\bibitem{ber} P. B\'{e}rard, In\'{e}galit\'{e}s isop\'{e}rim\'{e}triques et applications. Domaines nodaux des fonctions propres.
Goulaouic-Meyer-Schwartz Seminar, 1981/1982, Exp. No. XI, 10 pp., \'{E}cole Polytech., Palaiseau, 1982. 

\bibitem{berr} P. B\'{e}rard and B. Helffer, Remarks on the boundary set of spectral equipartitions, arXiv:1203.3566

\bibitem{blind} G. Blind, \"{U}ber Unterdeckungen der Ebene durch Kreise. J. Reine Angew. Math. 236 1969 145--173. 

\bibitem{blum} G. Blum, S. Gnutzmann and U. Smilansky, Nodal Domains Statistics: A Criterion for Quantum Chaos, Physical Review Letters 88 (2002), 114101.

\bibitem{bour} J. Bourgain, On Pleijel's Nodal Domain Theorem, IMRN 13 (2013), 1--7.

\bibitem{bra} L. Brasco, G. De Philippis and B. Velichkov, Faber-Krahn inequalities in sharp quantitative form, arXiv:1306.0392

\bibitem{caf} L. Cafferelli, F. H. Lin, An optimal partition problem for eigenvalues. J. Sci. Comput. 31 (2007), no. 1-2, 5--18. 

\bibitem{cour} R. Courant. Ein allgemeiner Satz zur Theorie der Eigenfunktionen selbstadjungierter Differentialausdr\"{u}cke, Nachr. Ges. G\"{o}ttingen (1923), 81-84.

\bibitem{fus} N. Fusco, F. Maggi and A. Pratelli, Stability estimates for certain Faber-Krahn, isocapacitary and Cheeger inequalities. Ann. Sc. Norm. Super. Pisa Cl. Sci. (5) 8 (2009), no. 1, 51--71. 

\bibitem{fus2} N. Fusco, F. Maggi and A. Pratelli, The sharp quantitative isoperimetric inequality. Ann. of Math. (2) 168 (2008), no. 3, 941--980. 

\bibitem{helffer} B. Helffer, On spectral minimal partitions: a survey. Milan J. Math. 78 (2010), no. 2, 575--590. 

\bibitem{heli} B. Helffer, T. Hoffmann-Ostenhof and S. Terracini, Nodal domains and spectral minimal partitions. Ann. Inst. H. Poincare Anal. Non Lineaire 26 (2009), no. 1, 101--138. 

\bibitem{jia} Jia-Chang Sun, On approximation of Laplacian eigenproblem over a regular hexagon with zero boundary conditions. Special issue dedicated to the 70th birthday of Professor Zhong-Ci Shi.
J. Comput. Math. 22 (2004), no. 2, 275--286. 

\bibitem{plei} A. Pleijel, Remarks on Courant's nodal line theorem, Comm. Pure Appl. Math. 9 (1956), 543--550.

\bibitem{polt} I. Polterovich, Pleijel's nodal domain theorem for free membranes. Proc. Amer. Math. Soc. 137 (2009), no. 3, 1021--1024. 

\end{thebibliography}
\end{document}